%% file: 2019_03_05.tex
\documentclass{amsart}

\usepackage[T1]{fontenc}
\usepackage[utf8]{inputenc}
\usepackage[USenglish]{babel}

\usepackage[font=small,labelfont=bf]{caption}

\usepackage[bookmarks=true,plainpages=false,linktocpage,colorlinks=true,citecolor=green!80!black,linkcolor=red!70!black,filecolor=magenta,urlcolor=magenta,breaklinks,pdfauthor={Thomas Jahn, Martin Winter},pdftitle={Vertex-Facets Assignments For Polytopes}]{hyperref}

\usepackage{amsmath,amsthm,amssymb}
\usepackage{calc, mathtools} 

\usepackage{tipa} 
\usepackage{colortbl,color,xcolor} 
\usepackage{centernot} 
\usepackage{array} 
\usepackage{enumitem,moreenum} 
\usepackage{cite} 
\usepackage[nameinlink,capitalize,noabbrev]{cleveref}
\usepackage[babel]{csquotes} 

\usepackage{tikz,pgfplots}

\newcommand{\FF}{\mathcal{F}}
\newcommand{\RR}{\mathbb{R}}                                     

\theoremstyle{plain}  
\newtheorem{Satz}{Theorem}
\newtheorem*{Satz*}{Theorem}
\newtheorem*{SatzA}{Theorem 6}
\newtheorem*{SatzB}{Theorem 7}
\newtheorem{Kor}[Satz]{Corollary}

\newtheorem{Prop}[Satz]{Proposition}

\newtheorem{Probl}[Satz]{Problem}

\theoremstyle{definition} 
\newtheorem{Def}[Satz]{Definition}
\newtheorem{Bsp}[Satz]{Example}

\newtheorem{Konstr}[Satz]{Construction}

\crefname{Satz}{Theorem}{Theorems}
\crefname{Prop}{Proposition}{Propositions}
\crefname{Lem}{Lemma}{Lemmas}
\crefname{Kor}{Corollary}{Corollaries}
\crefname{Bem}{Remark}{Remarks}
\crefname{Bsp}{Example}{Examples}
\crefname{Def}{Definition}{Definitions}
\crefname{Probl}{Problem}{Problems}
\crefname{Konstr}{Construction}{Constructions}

\DeclareMathOperator{\aff}{aff}
\DeclareMathOperator{\conv}{conv}

\newcommand\eset{\varnothing}
\newcommand{\card}[1]{\left\lvert#1\right\rvert}

\newcommand{\ddd}{...}

\newcommand{\shortStyle}{}

\newcommand{\ie}{\shortStyle{i.e.},}

\newcommand{\eg}{\shortStyle{e.g.}}

\newcommand{\resp}{resp.}

\let\dualsym=\Delta
\newcommand{\dual}[1]{{#1}^{\dualsym}}
\newcommand{\ddual}[1]{{#1}^{\dualsym\dualsym}}

\newcommand{\blank}{\,\cdot\,}

\newcommand{\join}{\bowtie}

\newcommand{\cupdot}{\mathbin{\mathaccent\cdot\cup}}

\newcommand{\mylabel}{$(\roman*)$}
\newenvironment{myenumerate}{\begin{enumerate}[label=\mylabel]}{\end{enumerate}}

\makeatletter
\renewcommand*{\eqref}[1]{%
  \hyperref[{#1}]{\textup{\tagform@{\ref*{#1}}}}%
}
\makeatother

\makeatletter
\@ifundefined{textcommabelow}{%
  \DeclareTextCommandDefault\textcommabelow[1]
    {\hmode@bgroup\ooalign{\null#1\crcr\hidewidth\raise-.31ex
     \hbox{\check@mathfonts\fontsize\ssf@size\z@
     \math@fontsfalse\selectfont,}\hidewidth}\egroup}%
}{}
\makeatother

\begin{document}

	
\title[Vertex-Facet Assignments For Polytopes]{Vertex-Facet Assignments For Polytopes}
	
\author[T. Jahn]{Thomas Jahn}
\address{Faculty of Mathematics, University of Technology, 09107 Chemnitz, Germany}
\email{thomas.jahn@mathematik.tu-chemnitz.de}
		
\author[M. Winter]{Martin Winter}
\address{Faculty of Mathematics, University of Technology, 09107 Chemnitz, Germany}
\email{martin.winter@mathematik.tu-chemnitz.de\newline\rule{0pt}{1.5cm}\includegraphics[scale=0.7]{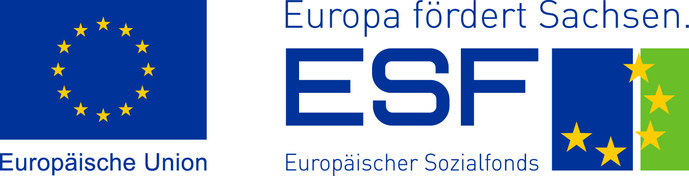}}
	
		
\date{\today}
\begin{abstract}
Motivated by the search for reduced polytopes, we consider the following question: 
For which polytopes exists a vertex-facet assignment, that is, a matching between vertices and non-incident facets, so that the matching covers either all vertices, or all facets?

We provide general conditions for the existence of such an assignment. We conclude that such exist for all simple and simplicial polytopes, as well as all polytopes of dimension $d\le 6$. 
We construct counterexample in all dimensions $d\ge 7$.
\end{abstract}

\maketitle


\section{Background and motivation}



The content of the present paper is motivated by the study of \emph{reduced convex bodies} as introduced by Heil \cite{Heil1978}.
A convex body is said to be reduced if no proper convex subset has the same minimum width as the body, \ie\ the same minimum distance between parallel supporting hyperplanes.
This notions is relevant, as the extremal bodies in several geometric minimization problems are necessarily reduced (\eg\ in Pál's problem \cite{Heil1978}).

One may attempt to find polytopal minimizers, called \emph{reduced polytopes}.
While the existence of reduced polytopes is clear in the Euclidean plane (\eg\ a regular triangle), the only known examples in higher dimensions were obtained in \cite{GonzalezJaPoWa2018} -- a single family of 3-dimensional polyhedra (see \cref{fig:reduced}).

A key concept in the study of reduced polytopes are different forms of \emph{antipodality} of faces.
For example, the authors of \cite[Theorem~4]{AverkovMa2008b} obtained the following necessary condition for a polytope $P\subset\RR^d$ to be reduced: for every vertex $v\in\FF_0(P)$ there is an antipodal facet $F\in\FF_{d-1}(P)$ in the following sense:
\begin{myenumerate}
	\item $P$ admits parallel hyperplanes supported at $v$ and $F$ respectively,
	\item the distance of the point of $F$ closest to $v$ is the minimum width of $P$, and
	\item the point of $F$ closest to $v$ is in the relative interior of $F$.
\end{myenumerate}
It is an immediate consequence of this result, that a reduced polytope admits an \emph{injective} map from its vertices to \emph{non-incident} facets.


This latter notion can be seen as a form of combinatorial antipodality and is the main topic of the present paper.
We investigate, whether this property can be used to narrow down the search for combinatorial types of polytopes that may admit a reduced realization.
It will be convenient to work with the following notion:

\begin{Def}
A \emph{vertex-facet assignment} of a polytope $P$ is a matching of non-incident vertex-facet pairs, that covers all vertices or all facets (or both).
\end{Def}

This notion works well with polytope duality: clearly, a polytope $P$ has a vertex-facet assignment if and only if $P$ or its dual $\dual P$ (or both) admit an injective map from vertices to non-incident facets.
Furthermore, it allows the straight forward application of graph theoretic methods (via the \emph{vertex-facet graph} in \cref{sec:graph-theory}), in particular, Hall's marriage theorem (\cref{res:hall}).

As an application of the latter, in \cref{sec:low_dim} we obtain sufficient and necessary conditions for the existence of vertex-facet assignments in terms of the $f$-vectors of $P$ and its faces.
These main results read as follows:

\begin{SatzA}
$P$ has a vertex-facet assignment if and only if
%
\begin{equation*}
f_0(\sigma)+f_0(\dual \sigma)\le \max\{f_0(P),f_0(\dual P)\}
\end{equation*}
for all faces $\sigma\in\FF(P)$.
\end{SatzA}

Here, $\dual P$ denotes the dual polytope, and $\dual\sigma$ the dual face to $\sigma\in\FF(P)$.
This notation is further explained in \cref{sec:notation}.

\begin{SatzB}
Suppose that for every face $\sigma\in\FF(P)$ holds: $\sigma$ or $\dual\sigma$ (or both) have at least as many facets as vertices.
Then $P$ has a vertex-facet assignment.
\end{SatzB}

It is a corollary of the second theorem that simple/simplicial polytopes always admit vertex-facets assignments (\cref{res:simple_simplicial}), as do polytopes in dimension $d\le 6$ (\cref{res:dim_6}).

In \cref{sec:high_dim} we construct polytopes without vertex-facets assignments in every dimension $d\ge 7$ using the free join construction (\cref{res:high-dimensions}).

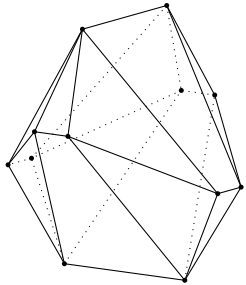
\begin{figure}
\centering
\input{img/valid1}
\caption{A reduced polyhedron.}
\label{fig:reduced}
\end{figure}

\section{A word on notation}\label{sec:notation}

Throughout this paper, let $P\subset\RR^d$ be a full-dimensional convex polytope, that is, $P$ is the convex hull of finitely many points, and $\aff(P)=\RR^d$.
Let $\FF(P)$ denote the face lattice of $P$, and $\FF_\delta(P)\subseteq\FF(P)$ the subset of faces of dimension $\delta$.

Our results are primarily combinatorial, and for the most part, we can identify $P$ with its face lattice $\FF(P)$.
Similarly, each face $\sigma\in\FF(P)$ is identified with the sub-lattice
$$\FF(\sigma):=[\eset,\sigma]:=\{\tau\in\FF(P)\mid \eset\subseteq\tau\subseteq\sigma\}\subseteq\FF(P).$$
The dual polytope $\dual P$ and, for each $\sigma\in\FF(P)$, the dual face $\dual\sigma$, shall be defined solely by their face~lattices:
the lattice $\FF(\dual P)$ is defined on the same set as $\FF(P)$, but with inverted lattice order.
The face lattice of $\dual\sigma$ is
$$\FF(\dual \sigma):=[P,\sigma]:=\{\tau\in\FF(\dual P)\mid P\supseteq \tau \supseteq \sigma\}\subseteq\FF(\dual P),$$
also with inverted lattice order.
Consider \cref{fig:lattice} for a visualization of the placement of these sub-lattices in $\FF(P)$.

It is important to distinguish between $\sigma\in\FF(P)$ as a face of $P$, and $\sigma$ as a polytope itself, since $\dual \sigma$ will refer to the dual polytope or the dual face, depending on the chosen interpretation.
Similarly, $\FF(\dual\sigma)$ will refer to distinct sub-lattices of $\FF(P)$ depending on the interpretation.
To avoid confusion, we therefore use the following convention: lower-case greek letters like $\sigma,\tau,\ddd$ do always denote faces of $P$, hence $\dual \sigma, \dual \tau, ...$ do always denote dual faces.
If other symbols are used to denote faces, this is explicitly mentioned in the respective place (this will only happen in \cref{sec:high_dim} in the proof of \cref{res:high-dimensions}).

Naturally, $\FF_0(P)$ refers to the set of vertices of $P$, and $\FF_0(\sigma)\subseteq\FF_0(P)$ refers to the set of those vertices of $P$, which are contained in the face $\sigma\in\FF(P)$. 
The notation $\FF_0(\dual P)$ is an alternative way to address the set of facets $\FF_{d-1}(P)$ of $P$.
While slightly unusual, we will use this notation in \cref{res:face_estimation} (and the proof of \cref{res:weak_condition}), to highlight the symmetry of the statement under polytope duality.
Similarly, $\FF_0(\dual\sigma)\subseteq\FF_0(\dual P)$ can be used to conveniently denote the set of those facets of $P$ that contain the face $\sigma\in\FF(P)$.

The cardinality of any of the sets defined above is denoted with $f_\delta(\blank):=\card{\FF_\delta(\blank)}$.

\begin{figure}
\includegraphics[width=0.55\textwidth]{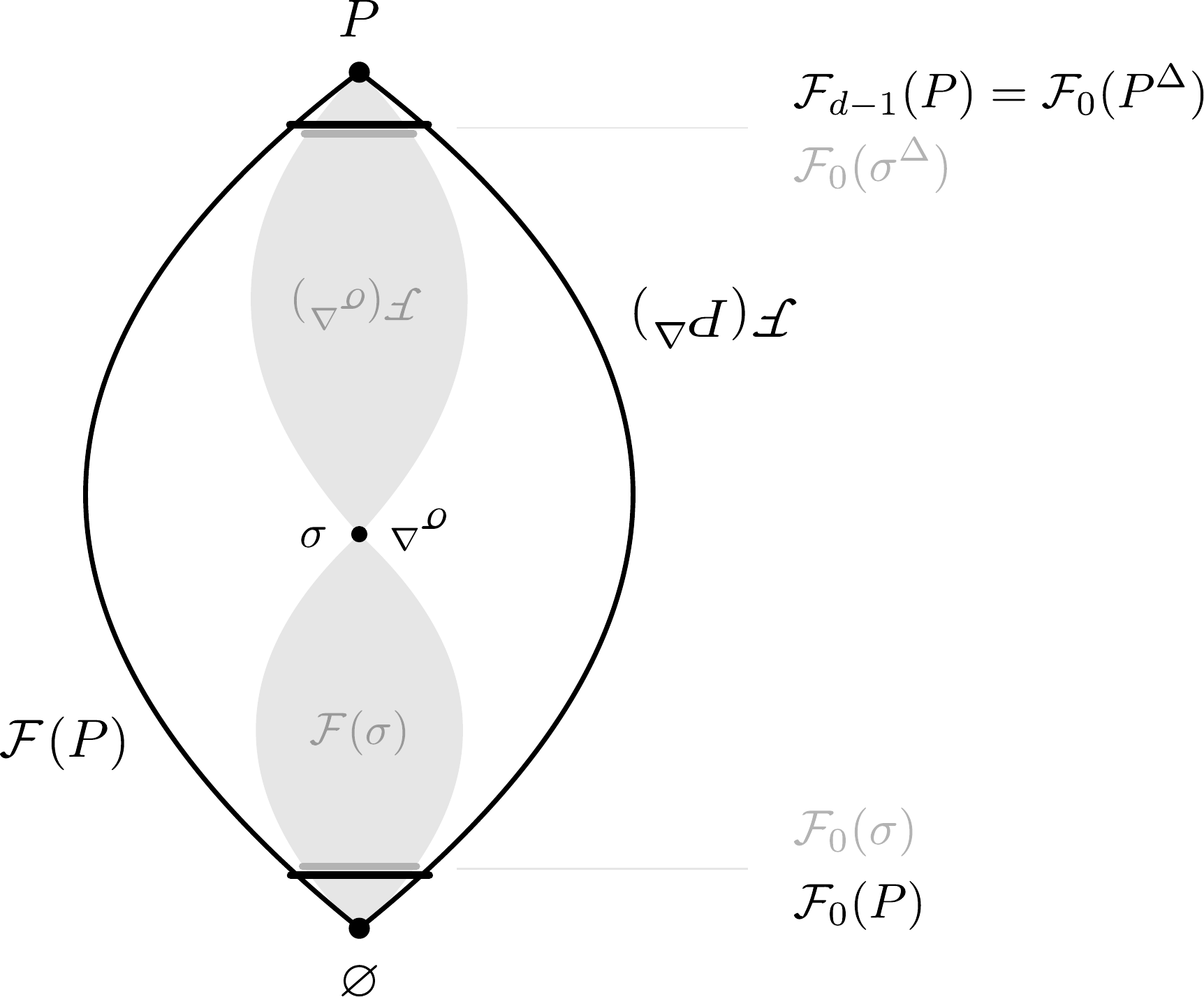}
\caption{
Schematic representation of the face lattice $\FF(P)$ of $P$, with the relevant sub-lattices and subsets highlighted.
The symbol $\sigma$ denotes a face $\sigma\in\FF(P)$. 
The upside-down labels indicate that the respective elements belong to the dual polytope and the respective lattices have inverted lattice order.
}
\label{fig:lattice}
\end{figure}

\section{Graph-theoretic formulation}
\label{sec:graph-theory}

%

To any polytope $P\subset\RR^d$, we assign a bipartite graph as follows:

\begin{Def}\label{def:vertex-facet-graph}
The \emph{vertex-facet graph} $G(P)$ of $P$ is the bipartite graph $G=(V_1\cupdot V_2,E)$ whose partition classes are (disjoint copies of)
the vertices $V_1=\FF_0(P)$ and the facets $V_2=\FF_{d-1}(P)$ of $P$. A vertex $v\in\FF_0(P)$ and a facet $F\in\FF_{d-1}(P)$ are adjacent in $G$ if and only if they are \emph{non-incident} in $P$.
\end{Def}

To avoid confusion with the terminology of polytopes, the elements of $V_1\cupdot V_2$ of $G(P)$ shall be called \emph{nodes} (instead of vertices).

A \emph{matching} $M\subseteq G(P)$ is a 1-regular subgraph. $M$ is said to \emph{cover} $\FF_0(P)$ (\resp\ $\FF_{d-1}(P)$), if~every vertex in $\FF_0(P)$ (\resp\ $\FF_{d-1}(P)$) is incident to an edge in $M$.

A vertex-facet assignment of $P$ is a matchings in the vertex-facet graph that covers either $\FF_0(P)$ or $\FF_{d-1}(P)$. 
Hall's marriage theorem \cite[\mbox{Theorem}~2.1.2]{Diestel2017} is a classical result that gives a necessary and sufficient condition for the existence of a matchings of a bipartite graph $G=(V_1\cupdot V_2,E)$ that covers, say, $V_1$.

\begin{Satz}\label{res:hall}
The bipartite graph $G=(V_1\cupdot V_2,E)$ possesses a matching that covers $V_1$ if and only if $\card{S}\leq\card{N_G(S)}$ for all $S\subseteq V_1$.
\end{Satz}

Here, $N_G(S)$ denotes the set neighbors of vertices in $S$.
The condition $\card{S}\leq\card{N_G(S)}$ for all $S\subseteq V_1$ in \cref{res:hall} is called \emph{Hall condition}.

\section{Existence in low dimensions}\label{sec:low_dim}

The existence of vertex-facet assignments are trivial for $d\in\{1,2\}$.
By applying Hall's marriage theorem (\cref{res:hall}) to the vertex-facet graph $G(P)$ (\cref{def:vertex-facet-graph}), we are able to show that there are no counterexamples up to dimension $6$ (\cref{res:dim_6}).
Furthermore, simple/simplicial polytopes have vertex-facet assignments in every dimension (\cref{res:simple_simplicial}).

\subsection{The case $f_0(P)\ge f_{d-1}(P)$}

In this section, we assume that $P$ has at least as many vertices as facets.
We therefore try to match vertices and facets in $G(P)$ in such a way, so that all facets are covered.

If $S\subseteq\FF_{d-1}(P)$ is a set of facets, there exists a nice geometric interpretation of the \emph{non-neighborhood} $\FF_0(P)\setminus N_{G(P)}(S)$.

\begin{Prop}\label{res:non_neighborhood}
The non-neighborhood of a set  $\{F_1,\ddd,F_k\}\subseteq\FF_{d-1}(P)$ of facets consists of the vertices of the face $F_1\cap\cdots \cap F_k$, \ie\
\begin{equation}\label{eq:non_neighborhood}
\FF_0(P)\setminus N_{G(P)}(\{F_1,\ddd,F_k\}) = \FF_0(F_1\cap\cdots\cap F_k).
\end{equation}
\begin{proof}
A vertex is in the non-neighborhood of $\{F_1,\ddd, F_k\}$ if and only if it (considered as a node in $G(P)$) is \emph{not} adjacent to any of the $F_i$. Since adjacency in $G(P)$ means non-incidence in $P$, these are exactly the vertices incident to \emph{all} the $F_i$, that is, contained in the face $F_1\cap\cdots \cap F_k$. The vertices of $P$ that are contained in the face $F_1\cap\cdots\cap F_k$ are exactly the vertices of this face.
\end{proof}
\end{Prop}

Together with Hall's marriage theorem (\cref{res:hall}) we immediately obtain

\begin{Kor}\label{res:facet_intersection}
A polytope $P$ with $f_0(P)\ge f_{d-1}(P)$ has a vertex-facet assignment if and only if for all $\{F_1,\ddd, F_k\}\subseteq \FF_{d-1}(P)$ holds
\begin{equation}\label{eq:facet_intersection}
f_0(F_1\cap\cdots\cap F_k)\le f_0(P)-k.
\end{equation}
%
\end{Kor}

\subsection{The general case}

We now drop the assumption $f_0(P)\ge f_{d-1}(P)$. 
The~result of \cref{res:facet_intersection} can be nicely symmetrized:


\begin{Satz}\label{res:face_estimation}
$P$ has a vertex-facet assignment if and only if
%
\begin{equation}\label{eq:face_estimation}
f_0(\sigma)+f_0(\dual \sigma)\le \max\{f_0(P),f_0(\dual P)\}
\end{equation}
for all faces $\sigma\in\FF(P)$.
\end{Satz}

Note that we wrote $f_0(\dual P)$ to denote the number of facets of $P$, and $f_0(\dual\sigma)$ to denote the number of those facets that contain $\sigma$.
This notation emphasizes the invariance of the result under polytope duality, \ie\ it holds for $P$ if and only if it holds for $\dual P$.

\begin{proof}[Proof of \cref{res:face_estimation}]
We can assume $f_0(P)\ge f_0 (\dual P)$, as otherwise we could~prove the statement for $\dual P$ instead. 
Under this assumption, \eqref{eq:face_estimation} becomes
\begin{equation}\label{eq:face_estimation_2}
f_0(\sigma)+f_0(\dual \sigma) \le f_0(P).
\end{equation}

Assume \eqref{eq:face_estimation_2} holds for all faces of $P$. Choose facets $F_1,\ddd,F_k\in\FF_{d-1}(P)$. The face $\sigma:=F_1\cap\cdots\cap F_k$ is contained in at least $k$ facets, \ie\ $f_0(\dual \sigma)\ge k$. We conclude
$$f_0(F_1\cap\cdots\cap F_k) = f_0(\sigma) \overset{\smash{\text{\eqref{eq:face_estimation_2}}}}\le f_0(P)-f_0(\dual \sigma) \le f_0(P)-k.$$
\Cref{res:facet_intersection} then yields the existence of a vertex-facet assignment.

Now conversely, assume there is a vertex-facet assignment. 
Let $\sigma\in\FF(P)$ be a face of $P$. 
Consider the set $\{F_1,\ddd,F_k\}:=\FF_0(\dual\sigma)$ of facets that contain $\sigma$. 
Then $\sigma=F_1\cap\cdots\cap F_k$ and $f_0(\dual \sigma)=k$. 
Since $P$ has a vertex-facet assignment, there holds \eqref{eq:facet_intersection} in \cref{res:facet_intersection}, and therefore
$$f_0(\sigma)+f_0(\dual \sigma)= f_0(F_1\cap\cdots\cap F_k)+ k\overset{\smash{\text{\eqref{eq:facet_intersection}}}}\le f_0(P).$$
This proves \eqref{eq:face_estimation_2}.
\end{proof}


We conclude a second sufficient condition, which is more convenient to apply to a large class of polytopes.

\begin{Satz}\label{res:weak_condition}
Suppose that for every face $\sigma\in\FF(P)$ holds: $\sigma$ or $\dual\sigma$ (or both) have at least as many facets as vertices.
%
Then $P$ has a vertex-facet assignment.
\begin{proof}

Choose some arbitrary face $\sigma\in\FF(P)$.
We show that $\sigma$ satisfies \eqref{eq:face_estimation}.
\Cref{res:face_estimation} then proves the existence of a vertex-facets assignment.

Suppose $\sigma$ has at least as many facets as vertices, or in formulas, $(*)\, k\ge f_0(\sigma)$, where $k$ denotes the number of facets of $\sigma$.
Each facet of $\sigma$ can be written as $\sigma\cap F$ with some $F\in \FF_0(\dual P)\setminus  \FF_0(\dual\sigma)$ (\ie\ $F$ is a facet of $P$ that does not contain $\sigma$).
It follows $(**)\,k\le f_0(\dual P)-f_0(\dual \sigma)$, and we conclude \eqref{eq:face_estimation} via
$$
f_0(\sigma)+f_0(\dual \sigma) 
\overset{\smash{(*)}}\le k+f_0(\dual \sigma)
\overset{\smash{(**)}}\le f_0(\dual P) 
\le \max\{f_0(P),f_0(\dual P)\}.
$$

On the other hand, if not $\sigma$, but only its dual face $\dual\sigma$ has at least as many facets as ver\-tices, we can apply above reasoning to $\dual\sigma$ as a face of $\dual P$, which yields
$$f_0(\dual\sigma) +f_0(\ddual \sigma) \le \max\{f_0(\dual P), f_0(\ddual P)\}.$$
As $\ddual P=P$ and $\ddual\sigma=\sigma$, this is equivalent to \eqref{eq:face_estimation} for $\sigma$.

\end{proof}
\end{Satz}

An immediate consequence of \cref{res:weak_condition} is the existence of vertex-facet assignments for polytopes up to dimension $6$.
\begin{Kor}\label{res:dim_6}
Every polytope $P\subset\RR^d$ of dimension $d\le 6$ possesses a vertex-facet assignment.
\end{Kor}
\begin{proof}
If $P$ is of dimension $d \le 6$, then for every face $\sigma\in\FF(P)$ we have
$$\min\{\dim(\sigma),\dim(\dual\sigma)\}\le 2.$$
Polytopes of dimension $\le 2$ have the same number of vertices and facets.
\Cref{res:weak_condition} then proves the existence.
\end{proof}

Therefore, the counterexamples can only occur in dimension $7$ or higher. Each $7$-dimensional counterexample must have a $3$-face $\sigma\in\FF(P)$ with $f_0(\sigma)+f_0(\dual \sigma)> \max\{f_0(P),f_0(\dual P)\}$.
We are going to construct such in the next section.

Another consequence is the existence of vertex-facet assignments for polytopes with only self-dual  facets, \eg\ simplicial polytopes:

\begin{Kor}\label{res:simple_simplicial}
Simple/simplicial polytopes have vertex-facet assignments.
\end{Kor}

\section{Counterexamples in high dimensions}\label{sec:high_dim}

In this section, we construct polytopes without vertex-facet assignments for any dimension $d\geq 7$.
These ``counterexamples'' are based on the \emph{free join} construction.

\begin{Konstr}
Given polytopes $P_i\subset\RR^{d_i},i\in\{1,2\}$, their free join $P_1\bowtie P_2$ is defined as the convex hull of copies of $P_1$ and $P_2$ that are embedded into skew affine subspaces of $\RR^{\smash{d_1+d_2+1}}$.
Lets call these copies $\bar P_1$ and $\bar P_2$.
%
%
%

The faces of the free join can be given in terms of the faces of $P_1$ and $P_2$.
Clearly, $\bar P_1$ and $\bar P_2$ are faces of the free join.
More generally, for each face $\sigma\in\FF(P_i)$, the corresponding face $\bar \sigma$ of $\bar P_i$ is a face of $P_1\join P_2$ as well. 
All other faces of the free join are of the following form:
for any two faces $\sigma_1\in \FF(P_1),\sigma_2\in \FF(P_2)$, the convex hull $\sigma_1\bowtie\sigma_2:=\conv(\bar \sigma_1\cup\bar\sigma_2)$ is a face of $P_1\join P_2$.
%

The number of vertices and facets of the free join are given as follows:
$$f_0(P_1\bowtie P_2)= f_0(P_1)+ f_0(P_2),\qquad f_{d-1}(P_1\bowtie P_2)= f_{d-1}(P_1)+ f_{d-1}(P_2).$$
%
This is clear for the vertices.
The facets of $P_1\bowtie P_2$ are $P_1\bowtie F,F\in\FF_{d-1}(P_2)$ and $F\bowtie P_2,F\in\FF_{d-1}(P_1)$.
For details regarding the free join and its properties, consider \cite[Corollary~2]{TiwaryEl2014}.
\end{Konstr}

\begin{Satz}\label{res:high-dimensions}
Let $P_1\subset\RR^{d_1}$ and $P_2\subset\RR^{d_2}$ be polytopes with
$$f_0(P_1) > f_{d-1}(P_1)\quad\text{and}\quad f_0(P_2) < f_{d-1}(P_2).$$
Then $P_1\join P_2$ does not have a vertex-facet assignment.
\begin{proof}
Let $P:=P_1\join P_2$.
The face $\bar P_1\in\FF(P)$ has $f_0(P_1)$ vertices and is contained in the $f_{d-1}(P_2)$ facets $P_1\bowtie F$ with $F\in\FF_{d-1}(P_2)$. 
We check that $\bar P_1$ violates condition \eqref{eq:face_estimation} (in the following, $\dual{\bar P_1}$ shall denotes a dual face in $P$, rather than the dual polytope of $\bar P_1$):
\begin{align*}
f_0(\bar P_1)+f_0(\dual{\bar P_1}) &= f_0(P_1)+f_{d-1}(P_2) \\
&> \max\{f_0(P_1)+f_0(P_2), f_{d-1}(P_1)+f_{d-1}(P_2)\} \\
&= \max\{f_0(P),f_{d-1}(P)\}. \\
&= \max\{f_0(P),f_0(\dual P)\}.
\end{align*}
Then, $P$ has no vertex-facet assignment by \cref{res:face_estimation}.
\end{proof}
\end{Satz}


The conditions $f_0(P_1) > f_{d-1}(P_1)$ and $f_0(P_2) < f_{d-1}(P_2)$ cannot be satisfied for polytopes in dimension $\leq 2$, which is the reason that this construction only yields counterexamples in dimensions $\geq 3+3+1=7$.

\begin{Bsp} We list some explicit counterexamples for $d\geq 7$.
\begin{enumerate}[label={(\roman*)},leftmargin=*,noitemsep]
\vspace{0.1em}
\item Counterexamples exist in all dimensions $d\ge 7$. 
Choose $d_1,d_2\ge 3$ with $d_1+d_2+1=d$ and polytopes $P_i\subset\RR^{d_i},i\in\{1,2\}$, so that $P_1$ has more vertices than facets (\eg\ the $d_1$-cube), and $P_2$ has more facets than vertices (\eg\ the $d_2$-dimensional cross-polytope). 
Then $P_1\bowtie P_2$ has no vertex-facet assignment by \cref{res:high-dimensions}.
\vspace{0.3em}
\item 
One might think that in self-dual polytopes the duality map $v\mapsto \dual v$ (dual face to the vertex $v$) somehow induces an injective map from vertices to non-incident facets, hence, that a self-dual polytope always has a vertex-facets assignment.
But surprisingly, there are self-dual polytopes \emph{without} vertex-facet assignments in dimension $d=2j+1,j\ge 3$. 
Let $P_1\subset\RR^j$ be a polytope with more vertices than facets (\eg\ the $j$-cube) and $P_2:=\dual P_1$ its dual.
The free join $P_1\bowtie P_2$ then is a self-dual $d$-dimensional polytope (see \cite[Corollary~2]{TiwaryEl2014}) without vertex-facet assignment (by \cref{res:high-dimensions}).
\end{enumerate}
\end{Bsp}

\section{Open problems and related questions}

Several related questions might be asked, for example, about the nature of further counterexamples. 
Are there other counterexamples for $d=7$ besides the free joins?
We suspect that such can be constructed by taking the convex hull of a free join and a point.
In general, what other ways are there to characterize polytopes without vertex-facet assignments?



Günter Ziegler brought up the following naturally related problem (personal communication):

\begin{Probl}\label{prob:ziegler}
Given a polytope $P\subset\RR^d$ with $f_0(P)\le f_{d-1}(P)$. Does there exist an injective map from vertices to \underline{incident} facets?
\end{Probl}

The answer to this seems comparatively easy and was already suggested by Ziegler himself: No. 
Even more, a polytope can have arbitrarily many more facets than vertices while still not having such an assignment. 
The following construction is based on an idea of Ziegler and uses the \emph{connected sum} operation (see \cite[Example~8.41]{Ziegler1995}).

\begin{figure}
\centering
\includegraphics[width=0.72\textwidth]{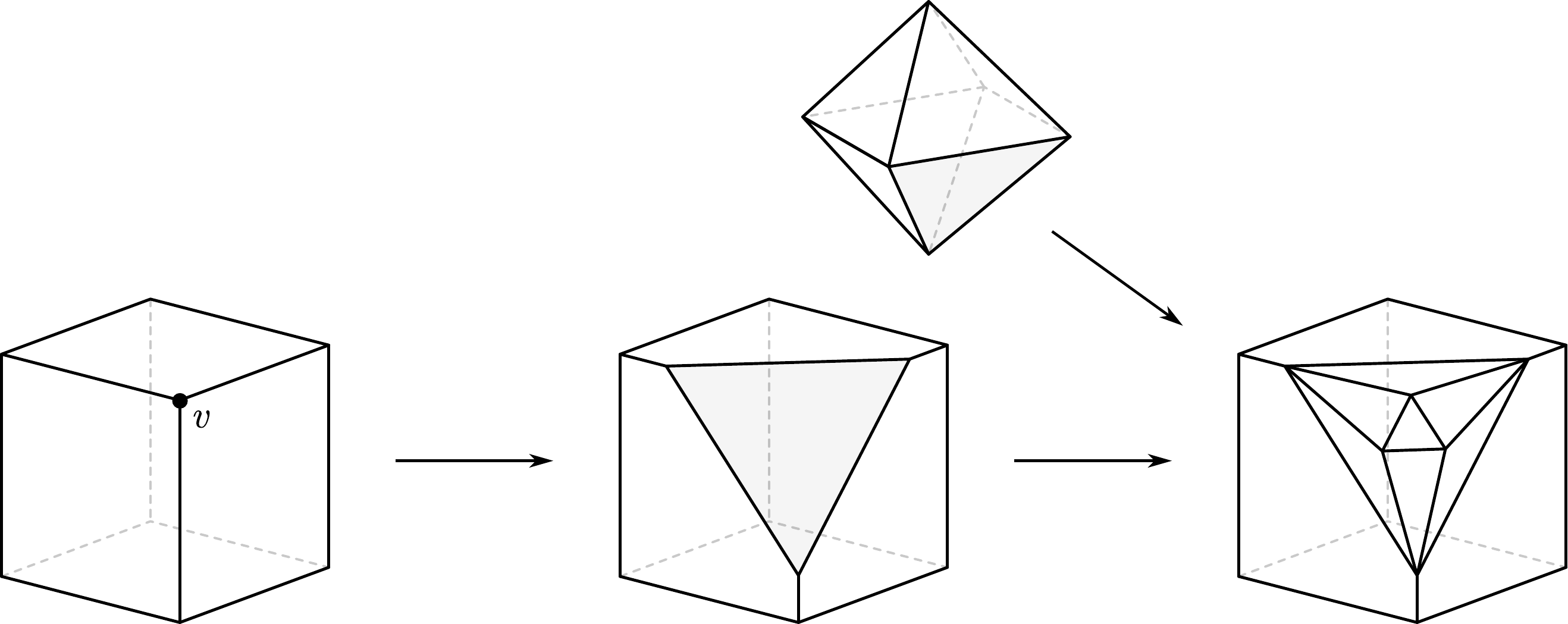}
\caption{Construction of a counterexample to the statement in \cref{prob:ziegler}.}
\label{fig:cube_mod2}
\end{figure}

\begin{Bsp}\label{ex:glue}
Take the $3$-cube $Q\subset\RR^3$ and let $v$ be one of its vertices. 
The other seven vertices are incident to a total number of only six facets, hence an injective map from vertices to incident facets cannot exist. 
One could argue, that this is due to $f_0(Q)>f_{d-1}(Q)$. 
However, one can modify $Q$ around $v$ without invalidating the current reasoning (see \cref{fig:cube_mod2}). 
Cut off the vertex $v$, and obtain a new triangular facet that is not incident to any of the other seven vertices. 
Choose a simplicial $3$-polytope $P\subset\RR^3$ with sufficiently many more facets than vertices and ``glue it'' (in the sense of the connected sum) to that triangular face.
The new polytope $Q':=P\# Q$ satisfies $f_0(Q')\ll f_{d-1}(Q')$ while still not having an injective map from vertices to incident facets.
\end{Bsp}


\par\medskip
\noindent
\textbf{Acknowledgements.} The second author gratefully acknowledges the support by the funding of the European Union and the Free State of Saxony (ESF).


\providecommand{\bysame}{\leavevmode\hbox to3em{\hrulefill}\thinspace}

\end{document}

%% file: img/valid1.tex
\begin{tikzpicture}[
		dot/.style={fill,shape=circle,scale=0.2},
		every label/.style={scale=0.7}
]
\begin{axis}[%
width=1.8in,
height=1.8in,
view={35}{10},
scale only axis,
hide axis,
xmin=-0.6,
xmax= 0.6,
ymin=-0.6,
ymax= 0.6,
zmin=-0.6,
zmax= 0.6,
]

\pgfmathsetmacro{\t}{0.550000000000000}
\pgfmathsetmacro{\x}{0.617649095979951}
\pgfmathsetmacro{\s}{0.135138493102611}
\pgfmathsetmacro{\h}{0.098430025240916}
\pgfmathsetmacro{\r}{0.354718358670911}

\coordinate (A) at (axis cs: \r,   0, -\t);
\coordinate (B) at (axis cs:-\r,   0, -\t);
\coordinate (C) at (axis cs:  0,  \r,  \t);
\coordinate (D) at (axis cs:  0, -\r,  \t);
\coordinate (E) at (axis cs: \h,  \x,  \s);
\coordinate (F) at (axis cs:-\h,  \x,  \s);
\coordinate (G) at (axis cs: \h, -\x,  \s);
\coordinate (H) at (axis cs:-\h, -\x,  \s);
\coordinate (I) at (axis cs: \x,  \h, -\s);
\coordinate (J) at (axis cs: \x, -\h, -\s);
\coordinate (K) at (axis cs:-\x,  \h, -\s);
\coordinate (L) at (axis cs:-\x, -\h, -\s);


\node [dot] at (A) {};
\node [dot] at (B) {};
\node [dot] at (C) {};
\node [dot] at (D) {};
\node [dot] at (E) {};
\node [dot] at (F) {};
\node [dot] at (G) {};
\node [dot] at (H) {};
\node [dot] at (I) {};
\node [dot] at (J) {};
\node [dot] at (K) {};
\node [dot] at (L) {};

\foreach \a/\b in {A/B,C/D,A/I,A/J,C/I,D/J,I/J,B/L,D/L,G/H,D/G,D/H,H/L,G/J,E/I,C/E,A/G,B/H} {
\edef\temp{\noexpand \draw (\a)--(\b);}
\temp
}

\foreach \a/\b in {B/K,E/F,C/F,F/K,A/E,B/F,C/K,K/L} {
\edef\temp{\noexpand \draw[dotted] (\a)--(\b);}
\temp
}
\end{axis}
\end{tikzpicture}